\documentclass[oneside,oldfontcommands]{article}
\usepackage[a4paper,margin=4cm]{geometry}
\usepackage[T1]{fontenc}
\usepackage{amsfonts,amsthm,amssymb,amsmath,mathdots, bbm,mathabx,mathrsfs,subcaption
}
\usepackage[linecolor=white,backgroundcolor=white,bordercolor=white,textsize=tiny]{todonotes}
\usepackage{amssymb, amscd}
\usepackage{color,graphicx}
\usepackage[all,cmtip]{xy}
\usepackage{lmodern}

\usepackage{silence}
\numberwithin{equation}{section} 
\title{Large permutation invariant random matrices are asymptotically free over the diagonal %(a.k.a. TIFA for Traffic Independence implies Freeness with Amalgamation)
}
\author{Benson Au\\\small Department of Mathematics\\\small University of California\\\small Berkeley, CA 94720-3840\and
Guillaume C\'ebron\\\small IMT; UMR 5219 \\\small Université de Toulouse; CNRS \\\small UPS, F-31400 Toulouse; France\and
Antoine Dahlqvist\\\small University College Dublin\\\small School of Mathematics and Statistics\\\small Belfield Dublin 4; Ireland\and
Franck Gabriel\\\small Imperial College London\\\small 
South Kensington Campus\\\small London SW7 2AZ; UK\and
Camille Male\\\small IMB; UMR 5251\\\small Université de Bordeaux; CNRS \\\small F-33405 Talence; France}
\date{\today}

\newtheorem{Th}{Theorem}[section]

\newtheorem{Def}[Th]{Definition}
\newtheorem{Prop}[Th]{Proposition}

\newtheorem{Lem}[Th]{Lemma}

\theoremstyle{remark}

\newtheorem{Rk}[Th]{Remark}

\renewcommand\leq\leqslant 
\renewcommand\geq\geqslant

\def\Tr{\mathrm{Tr}}

\def\esp{\mathbb E}

\def\etc{,\ldots ,}

\def\limN{\underset{N \rightarrow \infty}\longrightarrow}

\def\eps{\varepsilon}

\def\eq{\begin{eqnarray*}}
\def\qe{\end{eqnarray*}}
\def\eqa{\begin{eqnarray}}
\def\qea{\end{eqnarray}}

\def\mbf{\mathbf}
\def\mcal{\mathcal}

\def\mbb{\mathbb}

\def\mrm{\mathrm}

\def\bor{\begin{color}{orange}}
\def\eor{\end{color}}

\def\bte{\begin{color}{teal}}
\def\ete{\end{color}}
    
\begin{document}
\maketitle

\begin{abstract}
\noindent We prove that independent families of permutation invariant random matrices are asymptotically free over the diagonal, both in probability and in expectation, under a uniform boundedness assumption on the operator norm. We can relax the operator norm assumption to an estimate on sums associated to graphs of matrices, further extending the range of applications (for example, to Wigner matrices with exploding moments and so the sparse regime of the Erd\H{o}s-R\'{e}nyi model). The result still holds even if the matrices are multiplied entrywise by bounded random variables (for example, as in the case of matrices with a variance profile and percolation models).
\end{abstract}

\section{Introduction}

Non-commutative probability is a generalization of classical probability that extends the probabilistic perspective to non-commuting random variables. Following the seminal work of Voiculescu \cite{Voi91}, this setting provides a robust framework for studying the spectral theory of random multi-matrix models in the large $N$ limit. We outline the basic approach of this program as follows. 
\begin{enumerate}
\item In the non-commutative framework, Voiculescu's notion of \emph{free independence} plays the role of classical independence. This simple parallel yields a surprisingly rich theory, with free analogues of many classical concepts (e.g., the free CLT, free cumulants, free entropy, and conditional expectations). The scope of free probability further benefits from a robust analytic framework, allowing for the \emph{free harmonic analysis} \cite{VDN92}.
\item Free independence describes the large $N$ limit behavior of independent random matrices in many generic situations: for example, unitarily invariant random matrices \cite{Voi91} and Wigner matrices \cite{Dyk93}. The free probability machinery allows for tractable computations. In particular, we can compute the limiting spectral distributions of rational functions in these random matrices.
\end{enumerate}
Yet, many random matrix models of interest lie beyond the scope of free probability. One can accommodate such models by defining a new framework. This is the perspective of traffic probability \cite{Au16,Cebron2016,Gabriel2015c, Gabriel2015a, Gabriel2015b, MalePeche14, Male2011,Male17}.

\begin{enumerate}
\item The traffic framework enlarges the non-commutative probability framework. The notion of traffic distribution is richer and allows to consider a notion of independence which encodes the non-commutative notions of independences.
\item Permutation invariant random matrices provide a canonical model of traffic independence in the large $N$ limit. 
\end{enumerate}
The purpose of this article is to demonstrate that Voiculescu's notion of conditional expectation provides an analytic tool for traffic independence, in the context of large random matrices.

\subsection{Background}

We first recall the notion of freeness with amalgamation \cite{Voiculescu1986}.
\begin{Def}
\begin{itemize}
	\item An \emph{operator-valued probability space} is a triple $(\mcal A,\mcal B, E)$ consisting of a unital algebra $\mcal A$, a unital subalgebra $\mcal B\subset \mcal A$ and a \emph{conditional expectation} $E: \mcal A \to \mcal B$, i.e. a unit-preserving linear map such that $E(b_1ab_2) = b_1 E(a) b_2$ for any $a\in \mcal A$ and $b_1, b_2 \in \mcal B$.
	\item Let $K$ be an arbitrary index set. We write $\mcal B\langle X_k:k\in K\rangle $ for the algebra generated by freely noncommuting indeterminates $(X_k)_{k\in K}$ and $\mcal B$. For a monomial $P=b_0X_{k_1}b_1X_{k_2}\cdots X_{k_{n}}b_n\in \mcal B \langle X_k:k\in K\rangle$, the degree of $P$ is $n$ and its \emph{coefficients} are $b_0,\ldots, b_n$. The \emph{operator-valued distribution}, or $E$-distribution, of a family $\mbf A = (A^{(k)})_{k\in K}$ of elements in $\mcal A$ is the map of operator-valued moments
		$$E_{\mbf A}: P \in \mcal B\langle X_k:k\in K\rangle \mapsto E\big[P(\mbf A)\big] \in \mcal B.$$
	\item Sub-families $\mbf A_{1}=(A_{1}^{(k)})_{k\in K} \etc \mbf A_{L}=(A_{L}^{(k)})_{k\in K}$ of $\mathcal{A}$ are called \emph{free with amalgamation over $\mcal B$} (in short, free over $\mcal B$) whenever for any $n\geq 2$, any $\ell_1\neq \ell_2\neq \dots \neq \ell_n$ and any monomials $P_{1},\ldots,P_{n} \in \mcal B\langle X_k:k\in K\rangle$, one has 
	$$E \left[ \Big( P_{1}(\mbf A_{\ell_1}) - E \big[P_{1}(\mbf A_{\ell_1}) \big]\Big)  \cdots \Big( P_{n} (\mbf A_{\ell_n})- E \big[ P_{n} (\mbf A_{\ell_n})\big] \Big) \right]=0.$$
\end{itemize}
\end{Def}
	Ordinary freeness is the special case of operator-valued freeness when the algebra $\mcal B$ equals $\mbb C$. Operator-valued freeness works mostly like the ordinary one, and the freeness with amalgamation of $\mbf A_{1} \etc \mbf A_{L}$ allows to compute the operator-valued distribution of $\mbf A_{1}\cup\ldots \cup \mbf A_{L}$ from the separate operator-valued distributions.

For $N\in \mathbb{N}$, the set of $N\times N$ matrices is denoted by $\mathcal{M}_N$. We define the diagonal map $\Delta$ on $\mathcal{M}_N$ by $\Delta A= (\delta_{i,j} A(i,j))_{i,j=1}^N$ for any $A=(A(i,j))_{i,j=1}^N$. The set of $N\times N$ diagonal matrices is denoted by $\mathcal{D}_N$. Note that $(\mathcal{M}_N,\mathcal{D}_N,\Delta)$ is an operator-valued non-commutative probability space.

For random matrices, freeness with amalgamation over the diagonal first appeared in the work of Shlyakhtenko \cite{Shlyakhtenko1996} on Gaussian Wigner matrices with variance profile. Our motivation to consider this question for permutation invariant matrices is a recent result by Boedihardjo and Dykema \cite{BD17}. They proved that certain random Vandermonde matrices based on i.i.d. random variables are asymptotically $\mcal R$-diagonal over the diagonal matrices. These random matrices are invariant by right multiplication by permutation matrices. If this symmetry was with respect to multiplication by any unitary matrices, then we would have obtained convergence to ordinary $\mcal R$-diagonal elements \cite[Theorem 15.10]{NS}. This then suggests a link between permutation invariance and freeness with amalgamation over the diagonal.

\subsection{Asymptotic freeness with amalgamation over the diagonal}

In the sequel, when we speak about a family $\mbf A_N$ of $N\times N$ random matrices, we implicitly refer to a whole sequence $(\mbf A_N)_{N\in \mathbb{N}}$ of families of $N\times N$ random matrices $\mbf A_N = (A_N^{(k)})_{k\in K}$, where $K$ is independent of $N$. The family $\mbf A_N$ is said to be \emph{permutation invariant} if for any permutation $\sigma$ of the set $[N]:=\{1,\ldots,N\}$, the following two families of random variables have the same distribution:
$$\left(A_N^{(k)}(i,j)\right)_{k\in K,1\leq i,j\leq N}\overset{ \mcal Law}=\left(A_N^{(k)}(\sigma(i),\sigma(j))\right)_{k\in K,1\leq i,j\leq N}.$$
Equivalently, for any permutation matrix $S$ of size $N$, $$\left(A_N^{(k)}\right)_{k\in K}\overset{ \mcal Law}=\left(SA_N^{(k)}S^{-1}\right)_{k\in K}.$$
The following result is a simplified version of our main result (Theorem \ref{MainTh}).

\begin{Th}\label{FirstTh} Let $\mbf A_{N,1}=(A_{N,1}^{(k)})_{k\in K} \etc \mbf A_{N,L}=(A_{N,L}^{(k)})_{k\in K}$ be independent families of random matrices which are uniformly bounded in operator norm. Assume moreover that each family is permutation invariant. 

Then $\mbf A_{N,1} \etc \mbf A_{N,L}$ are \emph{asymptotically free with amalgamation} over $\mathcal{D}_N$ in the following sense. For any $n\geq 2$, any $\ell_1\neq \ell_2\neq \dots \neq \ell_n$ and any monomials $P_{N,1},\ldots,P_{N,n} \in \mathcal{D}_N\langle X_k:k\in K\rangle$ such that their degrees and the norms of their coefficients are bounded uniformly in $N$,  the matrix
	$$\eps_N := \Delta \left[ \big( P_{N,1}(\mbf A_{N,\ell_1}) - \Delta P_{N,1}(\mbf A_{N,\ell_1}) \big)  \cdots \big( P_{N,n} (\mbf A_{N,\ell_n})- \Delta  P_{N,n} (\mbf A_{N,\ell_n}) \big) \right]$$
converges to zero in Schatten $p$-norm for all $p\geq 1$. Namely, for any $p\geq 1$, we have $$ \esp\left[ \frac 1 N \Tr  \big[(\eps_N\eps_N^*)^p\big] \right] \limN 0.$$
\end{Th}

This convergence implies the convergence to zero in probability of the Schatten p-norms of $\eps_N$ as $N$ tends to infinity. Thus, independent large permutation invariant matrices are asymptotically free over the diagonal in probability in $(\mathcal{M}_N,\mathcal{D}_N,\Delta)$: 
for  a  typical  realization  of $\mbf A_{N,1} \etc \mbf A_{N,L}$, the  operator-valued distribution $\mbf A_{N,1}\cup \ldots \cup \mbf A_{N,L}$
is  close  to  the  distribution  of  copies  of  $\mbf A_{N,1} \etc \mbf A_{N,L}$ which  are  free  with  amalgamation  over $\mathcal{D}_N$.

In Theorem~\ref{MainTh} we first strengthen this result for the operator-valued space of random matrices generated by the $\mbf A_{N,\ell}$. In Section~\ref{Sec:Numerical} we also weaken the assumption of invariance by permutations: our results also apply if each matrix of $\mbf A_\ell^{(N)}$ is multiplied entrywise by bounded random variables.

\subsection{Numerical validation}

Free harmonic analysis in operator-valued spaces and Theorem \ref{MainTh} introduce new perspectives for the analysis of large random matrices. In particular one can use the numerical method  of Belinschi, Mai and Speicher \cite{Belinschi2017} to calculate the limiting empirical spectral distribution of rational functions random matrices.

We illustrate our result by computing the limiting distribution of the sum of two independent Hermitian matrices of various models: GUE matrices with variance profile, adjacency matrices of Erd\H{o}s-R\'enyi graphs, adjacency matrices of percolation on the cycle, diagonal matrices conjugated by a unitary Brownian motion or by the Fast Fourier Transform matrix. 

The rest of the article is organized as follows. In Section~\ref{Sec:Results}, we state our main results: Theorem~\ref{MainTh} and its generalizations Proposition~\ref{prop:boundsum} and Proposition~\ref{Prop:Hadamard}. In Section~\ref{Sec:Numerical}, we present an algorithm to compute the eigenvalue distribution of the sum of independent permutation invariant random matrices, and apply it to various models. Finally, Section~\ref{Sec:proofs} contains the proofs of the different results of Section~\ref{Sec:Results}.

\section{Statements of results}\label{Sec:Results}

\subsection{Freeness with amalgamation of large random matrices}
We now consider the operator-valued probability space of random matrices over a probability space. We denote it $(\mathcal{M}_N(L^{\infty}),\mathcal{D}_N(L^{\infty}),\Delta)$, where $\mathcal{M}_N(L^{\infty})$ is the space of $N\times N$ matrices of bounded random variables, and $\mathcal{D}_N(L^{\infty})$ is the space of $N\times N$ diagonal matrices of bounded random variables. In this context, $\mathcal{D}_N(L^{\infty})\langle X_k:k\in K\rangle$ is the space of monomials of \emph{random} diagonal matrices.

The minimal setting in order to formalize Theorem~\ref{FirstTh} in $(\mathcal{M}_N(L^{\infty}),\mathcal{D}_N(L^{\infty}),\Delta)$ requires the following definition. 

\begin{Def} Let $\mbf A_N$ be a family of random matrices. We define $\mcal A_N$ to be the smallest unital subalgebra containing the matrices of $\mbf A_N$ that is closed under $\Delta$, and we denote $\mathcal{B}_N = \Delta(\mcal A_N)$.
\end{Def}

Remark that $\mcal A_N=\mathcal{B}_N\langle \mbf A_N \rangle$, and $\mathcal{B}_N$ is the smallest subalgebra of $\mathcal{D}_N(L^{\infty})$ such that $\Delta\big(\mathcal{B}_N\langle \mbf A_N \rangle\big)\subset \mathcal{B}_N$. The triplet $(\mcal A_N, \mcal B_N, \Delta)$ is an operator-valued probability space. In order to formulate a notion of asymptotic freeness with amalgamation over $\mcal B_N$, the coefficients of the polynomials should be consistent as the dimension $N$ grows. The way we shall encode the coefficients independently of the dimension is by the following notion of graph monomial.

In this article, a {\em graph monomial} $g$  is the data of a finite, bi-rooted, connected and directed graph $G=(V,E,v_{in}, v_{out})$ and edge labels $\ell:E\to \{1,\ldots,L\}$ and $k:E\to K$. We allow multiplicity of edges and loops in our graphs. The roots $v_{in}, v_{out}$, are elements of the vertex set $V$, which we term the input and the output respectively. The roots are allowed  to be equal. The labels $\ell$ and $k$ allows to assign a matrix $A_{N,\ell(e)}^{(k(e))}$ to each edge $e\in E$.

\begin{figure}[h!]
 \centering
  \includegraphics[width=140pt]{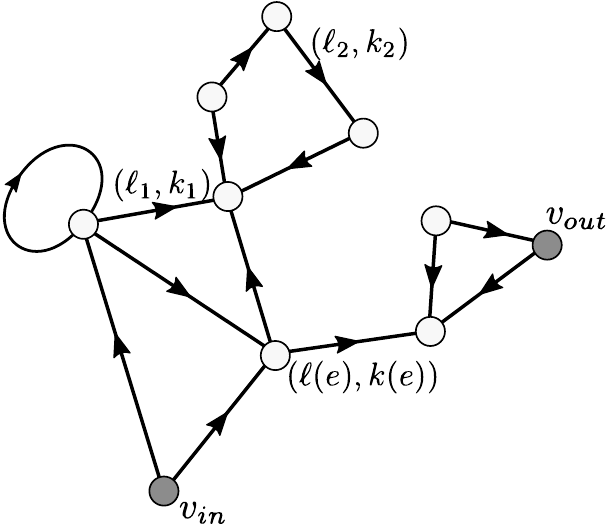}
 \caption{A graph mononial}
 \label{ex1}
\end{figure}

The evaluation of the graph monomial $g$ in the family of matrices $\mbf A_N = \mbf A_{N,1} \cup \dots \cup \mbf A_{N,L}$ is the $N\times N$ matrix $g(\mbf A_N)\in \mathcal{M}_N(L^{\infty})$  given by
$$g(\mbf A_N)(i,j):=\sum_{\substack{\phi:V\to [N]\\\phi(v_{in})=j, \phi(v_{out})=i }}\prod_{e=(v,w)\in E}A_{N,\ell(e)}^{(k(e))}(\phi(w),\phi(v)).$$
In Lemma \ref{Lem:FNAN}, we prove that $\mcal B_N$ is spanned by the matrices $g(\mbf A_N)$ where the graph of $g$ is a so-called cactus-type monomial, i.e. a graph monomial whose underlying graph is a \emph{cactus} and whose input and output are equal.

\begin{Th}\label{MainTh}  Let $\mbf A_{N,1}=(A_{N,1}^{(k)})_{k\in K} \etc \mbf A_{N,L}=(A_{N,L}^{(k)})_{k\in K}$ be independent families of random matrices that are uniformly bounded in operator norm. Assume moreover that each family, except possibly one, is permutation invariant. We write $\mbf A_N = \mbf A_{N,1} \cup \dots \cup \mbf A_{N,L}$ for the union of the families of matrices.

Then $\mbf A_{N,1} \etc \mbf A_{N,L}$ are \emph{asymptotically free with amalgamation} in the operator-valued probability space $(\mcal A_N, \mcal B_N, \Delta)$ in the following sense. Let $n\geq 2$, $\ell_1\neq \ell_2\neq \dots \neq \ell_n$ and $P_{N,1},\ldots,P_{N,n} \in \mathcal{B}_N\langle X_k:k\in K\rangle$ which are explicitly given by
$$P_{N,i}=g_{0,i}(\mbf A_N)X_{k_i(1)}g_{1,i}(\mbf A_N)\cdots X_{k_i(d_i)} g_{d_i,i}(\mbf A_N),$$
where each $g_{j,i}$ is a cactus-type monomial (which does not depend on $N$). Then, the matrix
	$$\eps_N := \Delta \left[ \big( P_{N,1}(\mbf A_{N,\ell_1}) - \Delta P_{N,1}(\mbf A_{N,\ell_1}) \big)  \cdots \big( P_{N,n} (\mbf A_{N,\ell_n})- \Delta  P_{N,n} (\mbf A_{N,\ell_n}) \big) \right]$$
 converges to zero in Schatten $p$-norm for all $p\geq 1$. Namely, for any $p\geq 1$, we have $$ \esp\left[ \frac 1 N \Tr  \big[(\eps_N\eps_N^*)^p\big] \right] \limN 0.$$
\end{Th}

 Note that there is no assumption of convergence of the $\Delta$-distribution of the matrices, as it is commonly assumed in the context of \emph{deterministic equivalents} \cite[Chapter 10]{Mingo2017}. 
\subsection{Generalizations}

In this section, we present some generalizations of Theorem \ref{MainTh}. First, we relax the bounded operator norm assumption.

\begin{Prop}\label{prop:boundsum}
 The conclusion of Theorem \ref{MainTh} remains valid if the operator norm bound assumption is replaced by the following weaker assumption: for any $\ell \in \{1,...,L\}$, any graph monomials $g_1,\ldots,g_n$ with equal input and output, 
\begin{equation}
\mathbb{E}\left[\prod_{i=1}^n \Tr \big( g_i(\mbf A_{N,\ell}) \big)\right]= O\big(N^{\sum_{i=1}^n\mathfrak f(g_i)/2}\big),\label{eq:MSbound}\end{equation}
where $\mathfrak f(g_i)\geq 2$ is determined from the forest $F$ of two-edge connected components of the graph underlying $g_i$ (Definition \ref{Def:TEC}). 
\end{Prop}

In particular, the conclusion of Proposition \ref{prop:boundsum} holds if
$$\mathbb{E}\left[\prod_{i=1}^n \Tr \big( g_i(\mbf A_{N,\ell}) \big)\right]= O\big(N^n\big),$$ which is the boundedness of the so-called traffic distribution \cite{Male2011}. For example, this holds for Wigner matrices with exploding moments as the matrix of the Erd\H{o}s-R\'enyi graph (see for instance \cite[Proposition 4.1]{Male17}), which do not satisfy the bounded operator norm property \cite[Proposition 12]{Zakharevich2006}.

The integer $\mathfrak f(g)$ appearing in the above theorem has been considered by Mingo and Speicher in \cite{Mingo2012}, where they proved that \eqref{eq:MSbound} is true if $\mbf A_{N,1} \etc \mbf A_{N,L} $ are uniformly bounded in operator norms (see Section~\ref{Sec:proofmain}).
\begin{Prop} \label{Prop:Hadamard} 
Let $\mbf A_{N,1} \etc \mbf A_{N,L} $ be independent families of random matrices that are uniformly bounded in operator norms, or which satisfies the bound \eqref{eq:MSbound}.  Assume moreover that each family is permutation invariant. Let $\big( \Gamma_{\ell}^{(k)} \big)_{\ell,k}$ be a family of random matrices with uniformly bounded entries, independent of $(\mbf A_{N,1} \etc \mbf A_{N,L})$. 
 Then the conclusion of Theorem \ref{FirstTh} and the conclusion of Theorem \ref{MainTh} remains true for the family $\tilde{\mbf A}_{N,1} \etc \tilde{\mbf A}_{N,L}$, where
	$$\tilde{\mbf A}_{N,\ell} = \left( A_{N,\ell}^{(k)} \circ \Gamma_{\ell}^{(k)} \right)_{k \in K},$$
and $\circ$ denotes the entrywise product.\end{Prop}
Let us emphasize that the families of matrices $\big( \Gamma_{1}^{(k)} \big)_{k} \etc \big( \Gamma_{L}^{(k)} \big)_{k} $ are not assumed to be neither independent nor permutation invariant, and so too for the $\tilde{\mbf A}_{N,1} \etc \tilde{\mbf A}_{N,L}$. In particular, Proposition \ref{Prop:Hadamard} can be applied to Wigner random matrices with variance profile \cite{Shlyakhtenko1996}.

\section{Examples and numerical validation}\label{Sec:Numerical}

In this section, we consider various models of independent permutation invariant random random matrices $X_N$ and $Y_N$. On the one hand, we construct the empirical spectral distribution of $X_N+Y_N$. One the other hand, we use the fixed point algorithm of Belinschi, Mai and Speicher \cite{Belinschi2017} to compute the spectral distribution of the free convolution with amalgamation over the diagonal of $X_N$ and $Y_N$. This depends only on the separated $\Delta$-distributions of $X_N$ and $Y_N$. Our main theorem guarantees that the difference between this two pictures becomes negligible in the limit in expectation and in probability. We actually obtain the numerical evidence for realizations of large matrices that asymptotic freeness over the diagonal holds.

\subsection{Amalgamated subordination property}
 
We need the following notion to explain the fixed point algorithm. Let $(\mathcal{A},\mathcal{D},\Delta)$ be an operator-valued probability space such that $\mcal A$ is a $\mcal C^*$-algebra. The operator-valued Stieltjes transform of a self-adjoint element $X$ of $\mcal A$ is the map
	$$ G_X: \Lambda \in \mcal D^+ \mapsto \Delta\big[ (\Lambda - X )^{-1}\big],$$
where $\mcal D^+$ is the set of elements $\Lambda\in \mcal D$ such that $\Im (\Lambda) := \frac{\Lambda-\Lambda^*}{2i} >0$. We also write $H_X$ for the map $H: \Lambda \mapsto G_X(\Lambda)^{-1}-\Lambda$. 

For random matrices, this map is then the diagonal of the resolvent of the matrices. Let us remark that this object was known to be an important tool in the analysis of large random matrices, for instance for random matrices with heavy tailed entries \cite{Bouchaud1994,Benarous2008,Bordenave2013,BDG}, for adjacency matrices of random graphs \cite{Khorunzhy2004}, or more recently for the local analysis of Wigner matrices with variance profile \cite{Ajanki2017}.

\begin{Th}[Theorem 2.2 of \cite{Belinschi2017}]\label{ThBel}
Two self-adjoint operator-valued random variables $X$ and $Y$ free with amalgamation over $\mathcal D$ satisfy the following \emph{subordination property}. For all $\Lambda\in \mcal D^+$, we have $G_{X+Y}(\Lambda) =G_X\big(\Omega(\Lambda)\big),$ where the subordination function $\Omega:\mathcal D^+ \to \mathcal D^+$ is the unique solution of the fixed point equation $\Omega(\Lambda)=F_\Lambda\Big( \Omega(\Lambda) \Big)$, where 
	$$F_\Lambda(\Omega) = H_Y\big(H_X (\Omega )+\Lambda \big)+\Lambda.$$
Moreover, it is the limit of the sequence $\Omega_n$ given by $\Omega_{n+1} = F_\Lambda(\Omega_n)$ for any $\Omega_0 \in \mcal D^+$.
\end{Th}

Let $X$ and $Y$ be two self-adjoint operator valued random variables, free with amalgamation. Assume that the $\Delta$-distributions of $X$ and $Y$ are those of our matrices $X_N$ and $Y_N$ respectively. The following algorithm generates an approximation of the value of the spectral density $g_{X+Y}(x)$ of $X+Y$ at $x \in \mbb R$ when it exists.
\begin{itemize}
\item[\textbf{Step 1}]Simulate a realization of $X_N$ and a realization of $Y_N$.
\item[\textbf{Step 2}]Set $\Lambda = (x + i y) \mbb I_N$ where $y>0$ is small. Compute the terms of the sequence $\Omega_1(\Lambda):=\Lambda$, $\Omega_{n+1}(\Lambda):=H_{Y_N}(H_{X_N}(\Omega_n(\Lambda))+\Lambda )+\Lambda$, $\forall n\geq 1$, until the difference between $\Omega_n$ and $\Omega_{n+1}$ is under a threshold. 
\item[\textbf{Step 3}]The value of the density $g_{X+Y}(x)$ is then close to
	$$\frac 1 \pi \Im \bigg( \frac{1}{N}\Tr\left[(\Omega_n(\Lambda)-{X_N})^{-1}\right]\bigg)$$
provided $y$ is small enough and the distribution admits a density at $x$.
\end{itemize}

\begin{Rk} Let us mention two strengthening of this method.
\begin{enumerate}
	\item The amalgamated $\mcal R$-transform of $Y$ is the map $\mcal D^+ \to \mcal D^+$ characterized by $G_Y(\Lambda)= \big(  \Lambda - R_Y\big( G_{Y}(\Lambda)\big) \big)^{-1}$. The function $\Omega(\Lambda)$ of Theorem \ref{ThBel} actually equals $\Lambda - R_Y\big( G_{X+Y}(\Lambda)\big)$. The knowledge of the amalgamated $\mcal R$-function of $X$ and/or $Y$ provides faster algorithms \cite[Theorem 11 of Chapter 9]{Mingo2017} which do not require the simulation of the matrices $X$ and/or $Y$. 
 	\item Thanks to a linearization trick, the fixed point algorithm described by Belinschi, Mai and Speicher can be extended in order to compute the distribution of any non-commuting rational function in $X$ and $Y$, see \cite[Theorem 2.2]{Belinschi2017}.
\end{enumerate}
 \end{Rk}

\subsection{Matrix models}

In the rest of the section, we will present examples of different models to illustrate of our results.

\begin{enumerate}
	\item We write $GUEvp(\eta)$ for the matrix with variance profile decomposed in blocks	
		$$GUEvp(\eta) =\sqrt{\frac 8{5\eta+3}} \left( \begin{array}{cc}
						\sqrt\eta X_{11} &    X_{12} \\
							  X_{21} &  \sqrt\eta X_{22} 
							\end{array} \right),$$
	where $\eta>0$ is a parameter, $X_{11} $ has size $N/4\times N/4$ and $\left( \begin{array}{cc}
						  X_{11} &    X_{12} \\
							  X_{21} &   X_{22} 
							\end{array} \right)$
		is a GUE matrix. 
	\item We write $ER(d)$ for the standardized adjacency matrix of a sparse Erd\H{o}s-R\'enyi graph, namely
		$$ER(d) = \frac {Y_N - d \mbb J_N}{\sqrt{d(1-d/(N-1))}},$$
	where 
\begin{itemize}
	\item $Y_N$ is a real symmetric random matrix, with null diagonal, and such that the strict subdiagonal entries are i.i.d. Bernoulli random variables with parameter $d/(N-1)$,
	\item $\mbb J_N$ is the matrix whose non diagonal entries are all $1/(N-1)$ and diagonal entries are zero,
	\item $d>0$ is a parameter.
\end{itemize}
	\item We write $Perm$ for the matrix $\frac 1 {\sqrt 2}(V+V' - 2 \mbb J_N)$, where $V$ is a uniform permutation matrix, and $Perm(p)$ for the matrix $\frac 1 {\sqrt p}Perm \circ \Gamma(p)$, where $\Gamma(p)$ is a symmetric matrix, with null diagonal, and whose strict subdiagonal entries are i.d.d. Bernoulli random variables with parameter $p$.
	\item For a given diagonal matrix $D$, we write $FFT_D $ for the matrix $V U D U ^*V ^*$ where $V $ is a uniform permutation matrix and $U$ is the \emph{Fast Fourier Transform} unitary matrix $\big(\frac 1{\sqrt N} \omega e^{\frac{-i 2\pi (n-1)(m-1)}N}\big)_{nm}$. We also denote $FFT_D(p)$ for the matrix $\frac 1 {\sqrt p}FFT_D \circ \Gamma(p)$, where $\Gamma(p)$ is as in the definition of $Perm(p)$. 
	\item For a given diagonal matrix $D$, we write $UMB_D(t)$ for the matrix $U_{t,N}DU_{t,N}^*$, where $U_{t,N}$ is a unitary Brownian motion starting from a uniform permutation matrix at time $t$. We refer to~\cite{Benaych2011} for the definition of the unitary Brownian motion and the related notion of $t$-freeness.
	\item We also define the three following diagonal matrix:
	\eq
		D1& := & diag(1, -1, 1, -1 \etc 1,-1),\\
		D2& := & diag\big(\underbrace{-1\etc -1}_{N/2}, \underbrace{1\etc 1}_{N/2}\big),\\
		D3 & := & \sqrt\frac{ 3}{14} diag\Big(  -2,-2+\frac2 N, -2+\frac4N \etc -2+\frac{N-1}N, \  1+\frac2N, 1+\frac4N \etc 2 \Big).
	\qe
	\end{enumerate}

\subsection{Experiment framework and comments}
Figure \ref{Fig:NumResult} reports the result of 8 numerical simulations, for different models for $X_N$ and $Y_N$ indicated in the legend of the pictures.

\begin{figure}
\medskip
\begin{subfigure}{0.48\textwidth}
\includegraphics[width=\linewidth]{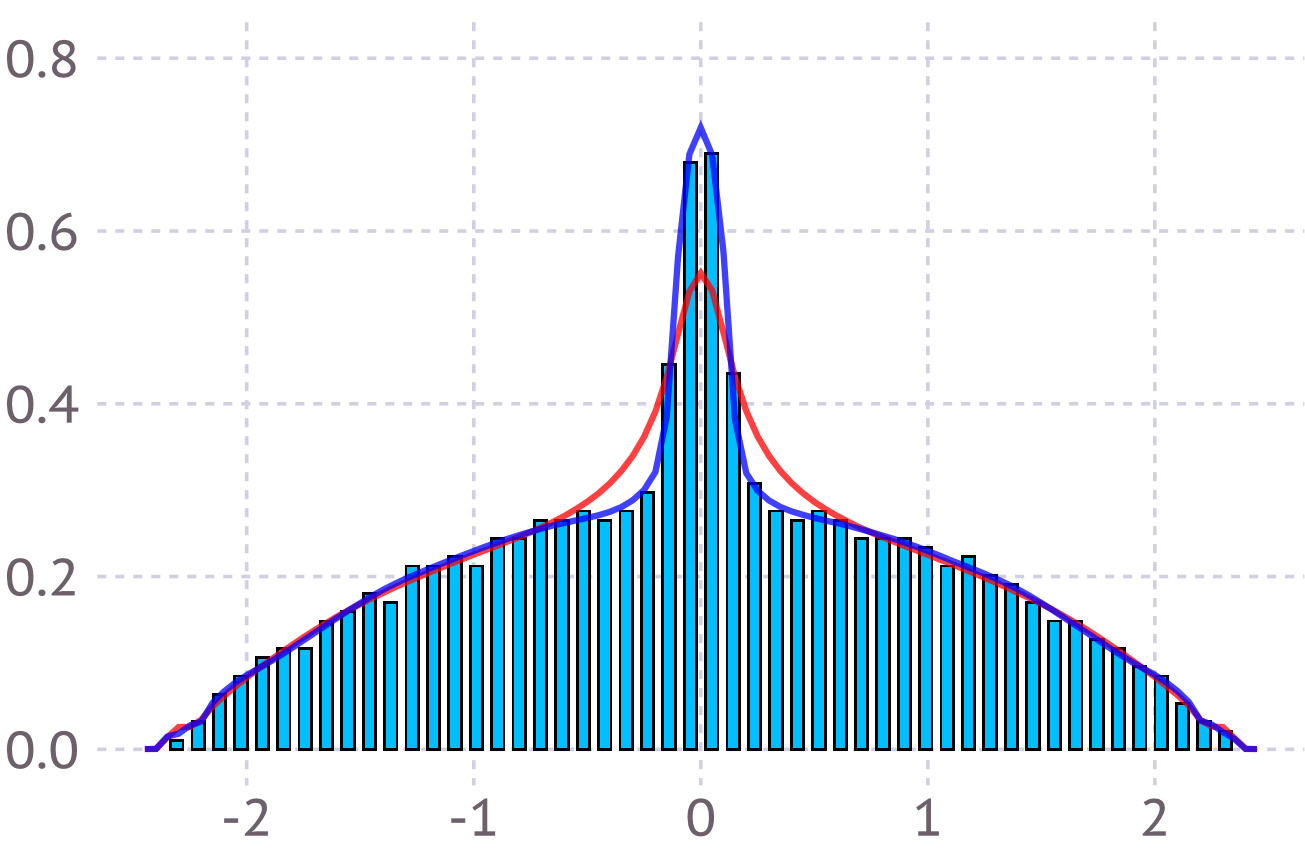}
\caption{ GUE(0.03125) and ER(1)}  
\end{subfigure}\hspace*{\fill}
  \hspace*{\fill}
\begin{subfigure}{0.48\textwidth}
\includegraphics[width=\linewidth]{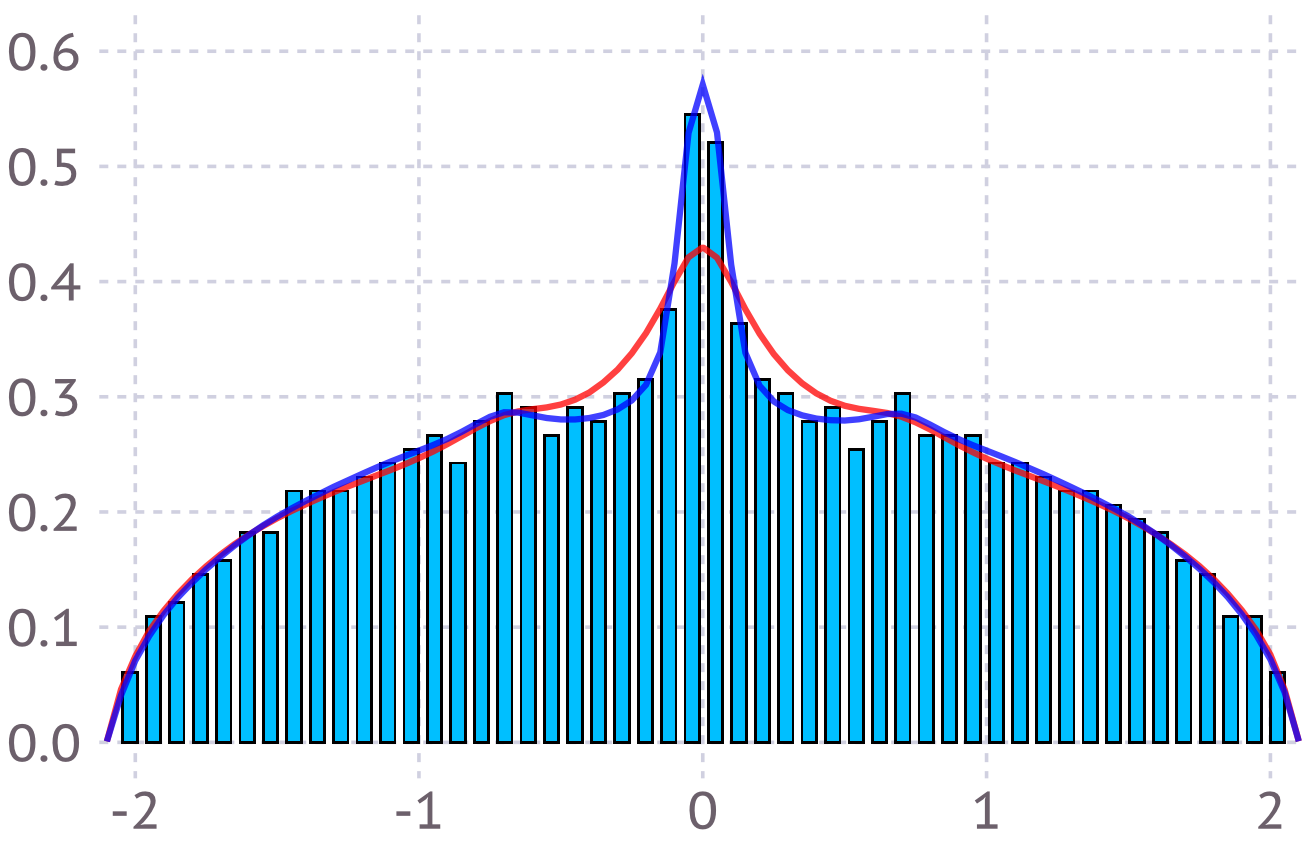}
\caption{GUE(0.03125) and Perm(0.5)} 
\end{subfigure}
\begin{picture}(0.3,0.1)
     \put(-355,63){\includegraphics[width=.45\linewidth]{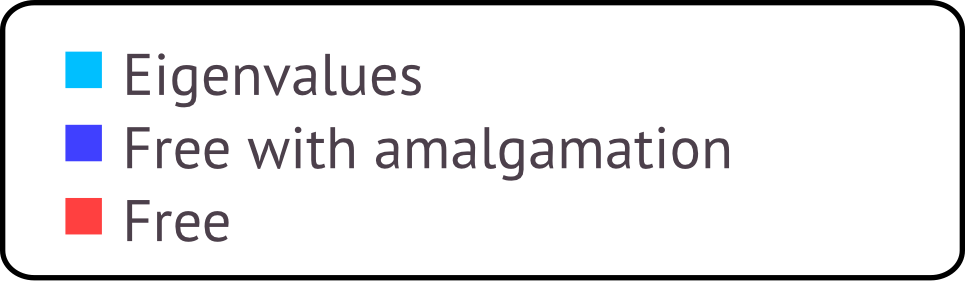}}
  \end{picture}

\medskip
\begin{subfigure}{0.48\textwidth}
\includegraphics[width=\linewidth]{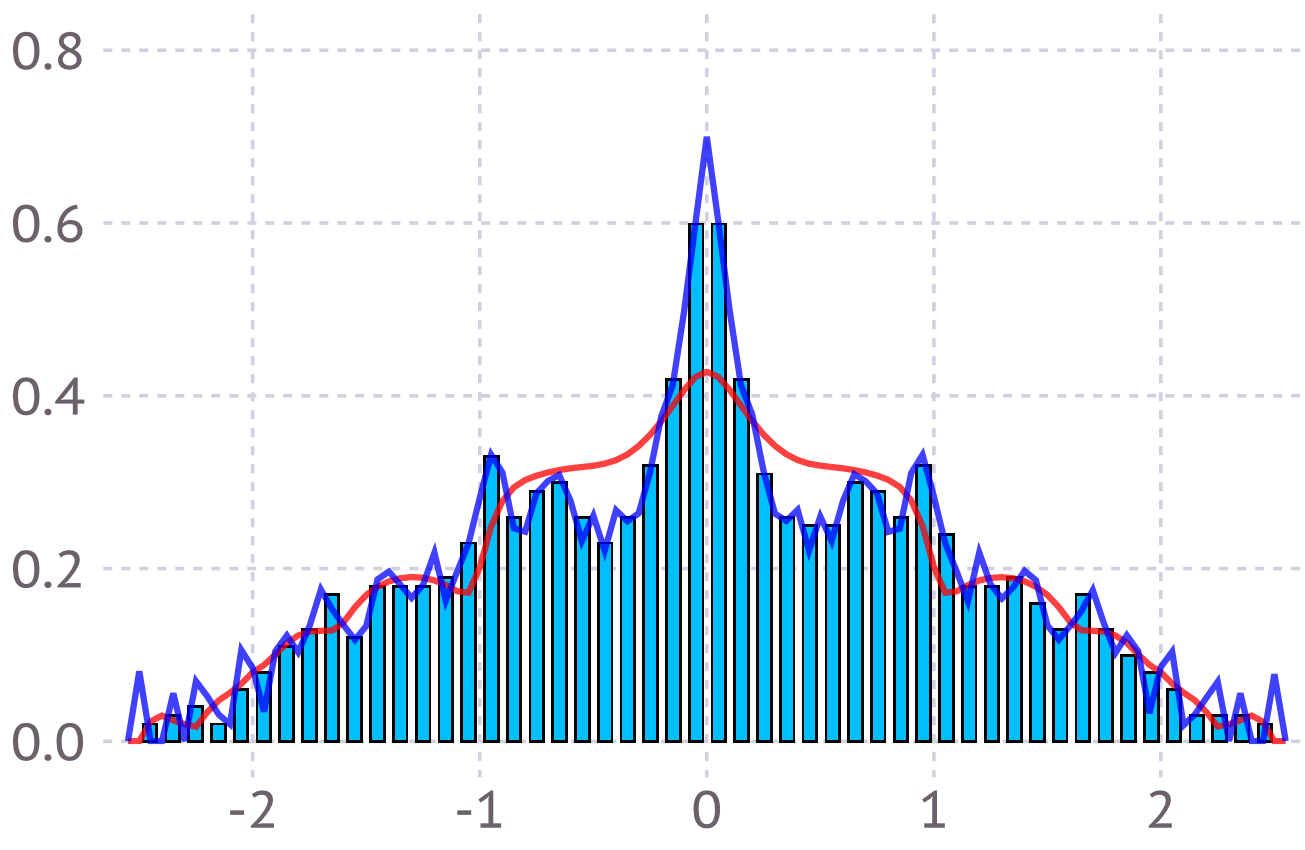}
\caption{FFT${}_{D2}$(0.5) and ER(1)}  
\end{subfigure}\hspace*{\fill}
\begin{subfigure}{0.48\textwidth}
\includegraphics[width=\linewidth]{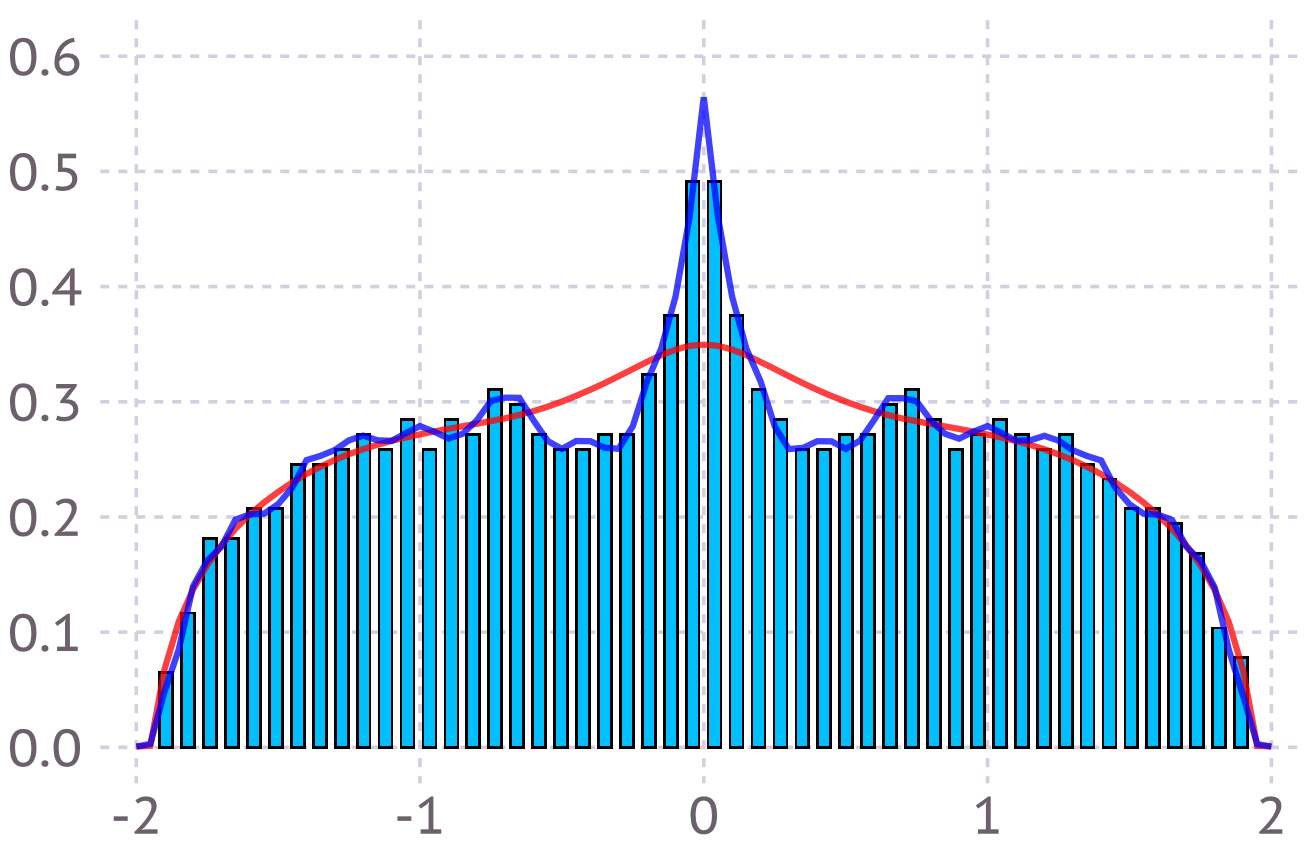}
\caption{FFT${}_{D2}$(0.5) and Perm(0.5)}  
\end{subfigure}

\medskip
\begin{subfigure}{0.48\textwidth}
\includegraphics[width=\linewidth]{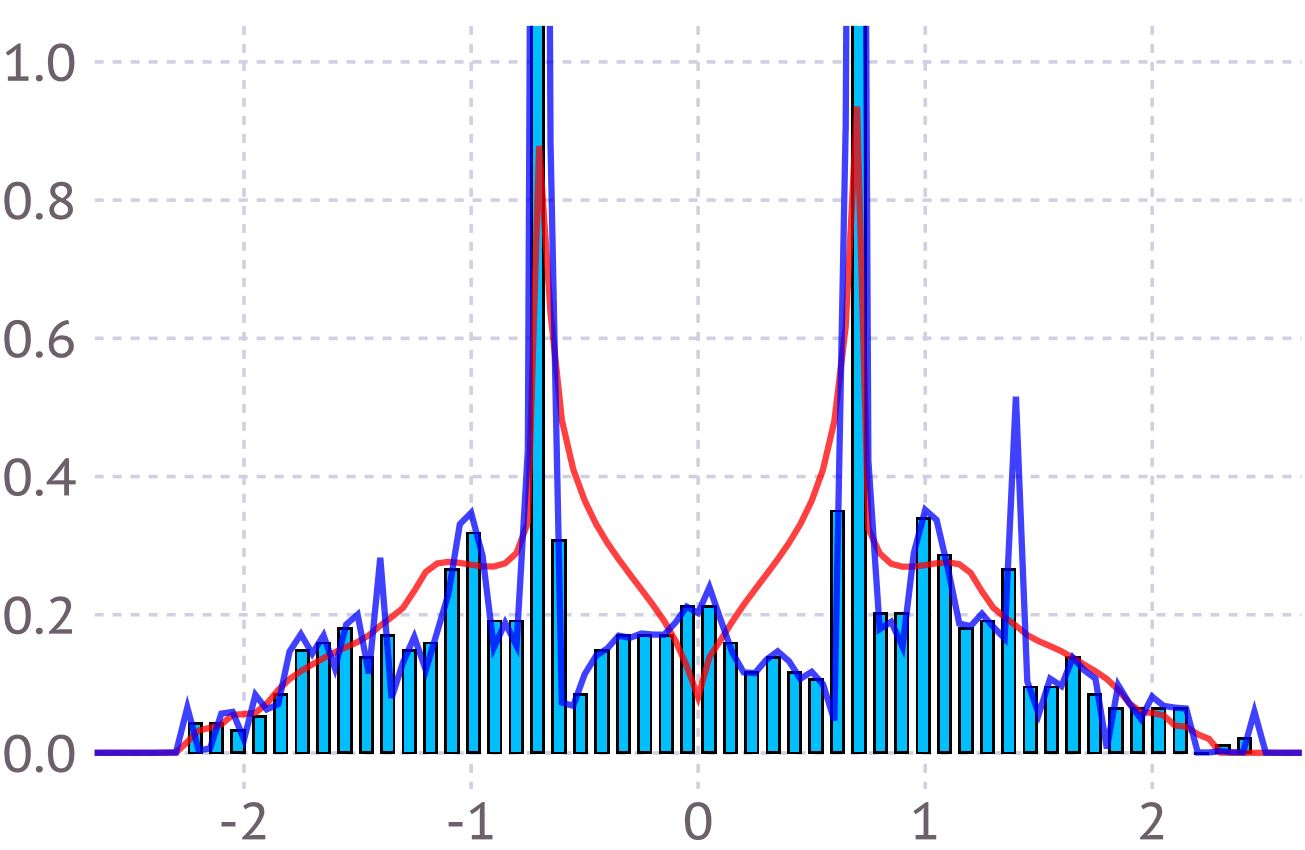}
\caption{UBM${}_{D2}$(0.125) and ER(1)}  
\end{subfigure}\hspace*{\fill}
\begin{subfigure}{0.48\textwidth}
\includegraphics[width=\linewidth]{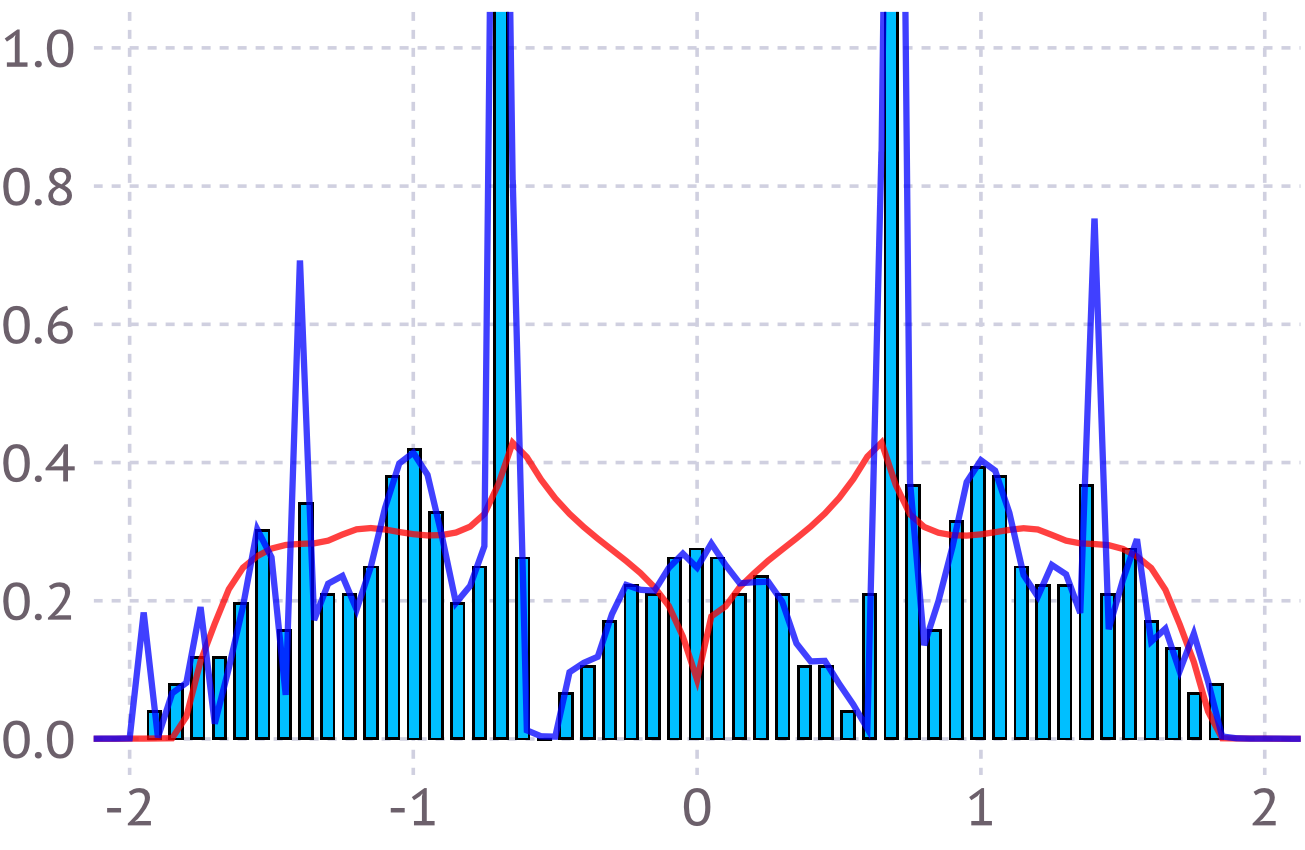}
\caption{UBM${}_{D2}$(0.125)  and Perm(0.5)}  
\end{subfigure}

\medskip
\begin{subfigure}{0.48\textwidth}
\includegraphics[width=\linewidth]{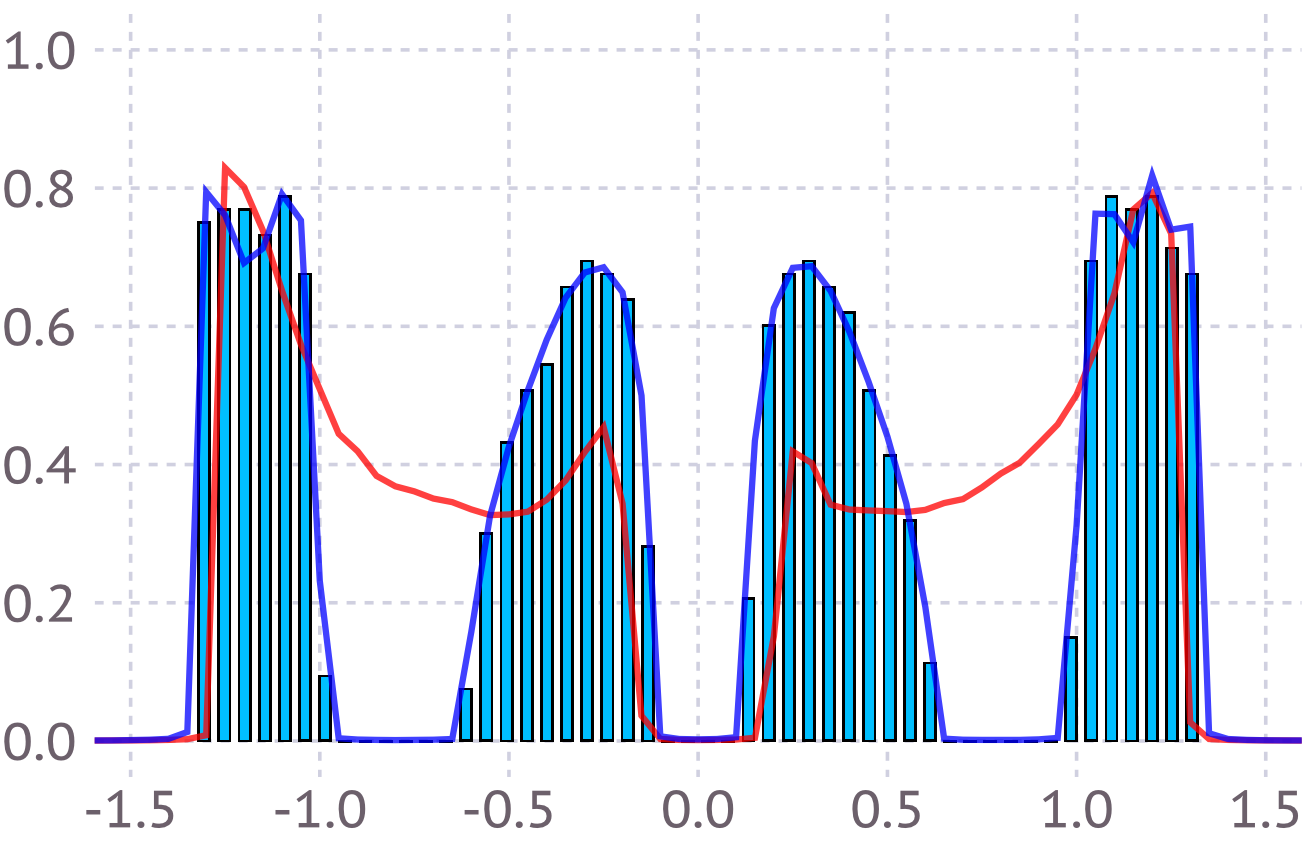}
\caption{UBM${}_{D1}$(0.125) and UBM${}_{D3}$(0.125)} 
\end{subfigure}\hspace*{\fill}
\begin{subfigure}{0.48\textwidth}
\includegraphics[width=\linewidth]{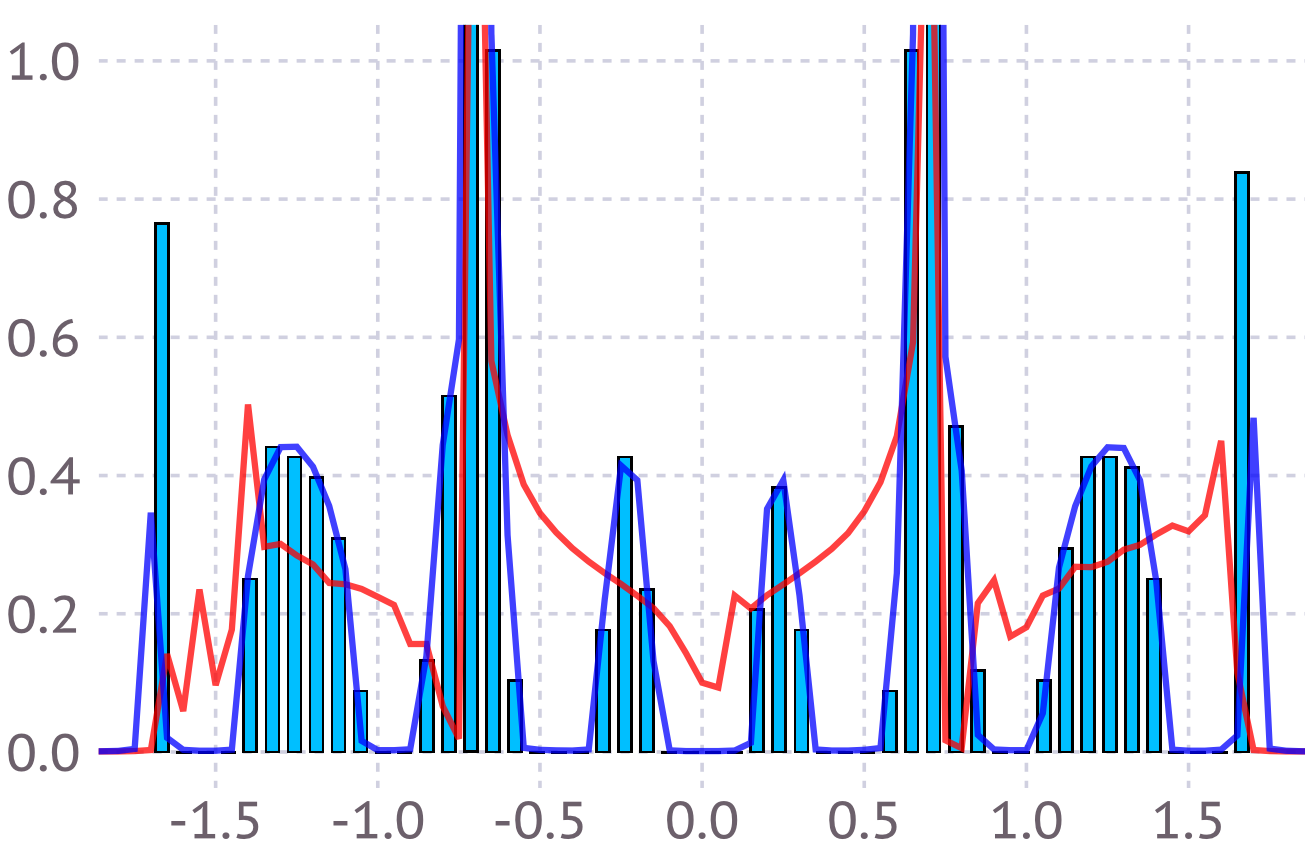}
\caption{UBM${}_{D2}$(0.125) and FFT${}_{D1}$(0.5)}  
\end{subfigure}
\caption{Results of the numerical simulations.} \label{Fig:NumResult}
\end{figure}

In each case, we represent

\begin{itemize}
	\item the histogram of the eigenvalues of one realization of $\frac 1 {\sqrt 2}(X_N+Y_N)$ for matrices of size $N=1000$,
	\item in blue, the density of the spectral distribution of $\frac 1 {\sqrt 2}(X+Y)$, where $X$ and $Y$ are free over the diagonal and such that their $\Delta$-distribution is the one of $X_N$ and $Y_N$ 
	\item in red, the density of the free convolution of the empirical eigenvalues distributions of $X_N$ and $Y_N$. 
\end{itemize}

In the algorithm, the small parameter $y$ equals $0.001$. We stop the fixed point process at the threshold $\| \Omega_n - \Omega_{n+1}\|_\infty \leq 0.001$.

The prediction by asymptotic freeness over the diagonal fits accurately the histogram in all situations we have investigated numerically. In Picture (a) to (d) we see that the deviation between the free case and the free with amalgamation case is mainly at the center of the spectrum. In Picture (e) to (h), the time of the Brownian motion is quite small and so the difference between the red and the blue line is more important.

\section{Proof of the theorems}\label{Sec:proofs}

\subsection{Algebras of graph polynomials}\label{Sec:CactiField}

We first make precise the notion of cacti mentioned previously.

\begin{Def}
A connected graph is a \emph{cactus} if every edge belongs to exactly one simple cycle. An \emph{oriented cactus} is a cactus such that each simple cycle is directed.
\end{Def}

\begin{figure}[h!]
 \centering
  \includegraphics[width=80pt]{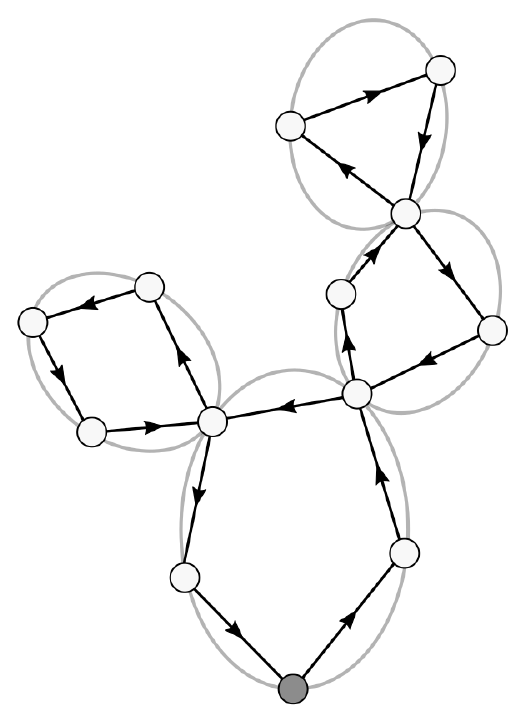}
 \caption{Cactus with four simple cycles}
 \label{ex2}
\end{figure}

We define $\mathcal{C}$ to be the set of \emph{cactus-type monomials}, that is such that the underlying graph is an oriented cactus and the input and output are equal. We denote by $\mcal C_N$ the vector space generated by $\left(g(\mbf A_N)\right)_{g\in\mathcal{C}}$. More generally we define $\mcal F$ to be the set of graph monomials obtained by considering a directed simple path oriented from the right to the left and by attaching oriented cacti on the vertices of the path. The input is the far right vertex of the path, the output is the far left. We then denote by $\mcal F_N$ the vector space generated by $\left(g(\mbf A_N)\right)_{g\in\mathcal{F}}$. It is clear that $\mcal C \subset \mcal F$ and so $\mcal C_N \subset \mcal F_N$. Moreover, note that $\mcal F_N = \mcal C_N\langle \mbf A_N \rangle$, namely it is the vector space generated by elements of the form 
\begin{align}
\label{eq:monomial}
P=g_0(\mbf A_N) A_{N,\ell_{1}}^{(k_{1})} g_1(\mbf A_N) \cdots A_{N,\ell_{d}}^{(k_{d})}g_d(\mbf A_N),
\end{align}
where, for $i\in\left\{ 0,...,d\right\}$, $g_{i}$ is in $\mcal C$.

\begin{Lem}\label{Lem:FNCN} The spaces $\mcal C_N$ and $\mcal F_N$ are unital algebras and $\Delta(\mcal F_N) = \mcal C_N$. 
\end{Lem}

\begin{proof} For any graph monomials $g_{1}$ and $g_{2}$, the product of matrices $g_1(\mbf A_N) g_2(\mbf A_N)$ is equal to $g(\mbf A_N)$, where $g$ is obtained by identifying the input of $g_1$ with the output of $g_{2}$. It follows that $\mcal C_N$ is an algebra, and so is $\mcal F_N$. Both are unital by considering the trivial graph with a single isolated vertex. 

For any graph monomial $g$, the matrix $\Delta\big( g(\mbf A_N) \big)$ is equal to $\Delta(g)(\mbf A_N)$, where $\Delta(g)$ is obtained by identifying the input with the output of $g$. It follows that $\Delta(\mcal F_N) = \mcal C_N$.
\end{proof}

\begin{Lem}\label{Lem:FNAN} One has $\mcal F_N = \mcal A_N$ and consequently $\mcal C_N = \mcal B_N$. 
 \end{Lem}
 
\begin{proof} Recall that $\mcal A_N$ is the smallest unital algebra containing $\mbf A_N$  stable under $\Delta$. It is clear that $\mcal A_N \subset \mcal F_N$ since $\mbf A_N \subset \mcal F_N$ and $\Delta(\mcal F_N) \subset \mcal F_N$ by Lemma \ref{Lem:FNCN}.

We now prove the reverse inclusion by induction on the number of edges. If $g$ has a single edge then $g(\mbf A_N) \in{\mbf A_N} \cup \Delta(\mbf A_N) \subset\mcal A_N$.

 Assume that for some $n\geq 2$, every $g(\mbf A_N) \in \mcal F_N$ whose graph has fewer than $n$ edges belongs to $\mcal A_N$. Let $g(\mbf A_N) \in \mcal F_N$ for some graph with $n$ edges. If the input and the output of $g$ are not equal, or if they are equal but belong to more than one cycle, then we can write $g(\mbf A_N)= g_1(\mbf A_N)g_2(\mbf A_N)$, where $g_1$ and $g_2$ have fewer than $n$ edges. Since $\mcal A_N$ is an algebra, then $g(\mbf A_N)\in \mcal A_N$.

Assume that the input and the output are equal to a vertex $v$ and belong to exactly one cycle. Then one can split the vertex $v$ into two vertices $v_{in}$ and $v_{out}$, allowing one to exhibit a product $P$ of the form (\ref{eq:monomial}) such that $\Delta\left(P\right)=g(\mbf A_N)$. Now we can factorize $P = g_1(\mbf A_N) g_2(\mbf A_N)$ as in the previous case, where $g_1(\mbf A_N)$ and $g_2(\mbf A_N)$ belong to $\mcal A_N$ by the induction hypothesis. Hence $g(\mbf A_N) = \Delta\big( g_1(\mbf A_N) g_2(\mbf A_N) \big)$ belongs to $\mcal A_N$.
\end{proof}

\subsection{Preliminary lemmas}\label{Sec:PrelLem}

We start by considering arbitrary families $\mbf A_{N,1} \etc \mbf A_{N,L}$ of random matrices. By enlarging them if necessary, we can assume that the families are closed under $*$. 

\begin{Lem}\label{MainLem0}If
$$\esp\left[ \frac 1 N \Tr\,  \eps_N\right]\limN 0$$
for every $\eps_N$ as in Theorem~\ref{MainTh}, i.e. $$\eps_N:=\big( P_{N,1}(\mbf A_{N,\ell_1}) - \Delta P_{N,1}(\mbf A_{N,\ell_1}) \big) \cdots \big( P_{N,n}(\mbf A_{N,\ell_{n}}) - \Delta P_{N,{n}}(\mbf A_{N,\ell_{n}}) \big)$$ where
 $n\geq 2$, $\ell_1\neq \ell_2\neq \dots \neq \ell_n$ and $P_{N,1},\ldots,P_{N,n} \in \mathcal{B}_N\langle X_k:k\in K\rangle$ are explicitly given by
$P_{N,i}=g_{0,i}(\mbf A_N)X_{k_i(1)}g_{1,i}(\mbf A_N)\cdots X_{k_i(d_i)} g_{d_i,i}(\mbf A_N),$
with $g_{j,i}$ cactus-type monomials (which does not depend on $N$), then,
we also have
$$\esp\left[ \frac 1 N \Tr (\eps_N\eps_N^*)^p \right]\limN 0$$
for any such $\eps_N$.

\end{Lem}

\begin{proof} Let us remark that $(\eps_N\eps_N^*)^p = \eps_N\eps_N^*(\eps_N\eps_N^*)^{p-1}$ can be written as: 
\begin{align*}
\big( P_{N,1}(\mbf A_{N,\ell_1}) - \Delta P_{N,1}(\mbf A_{N,\ell_1}) \big) \cdots \big( P_{N,n}(\mbf A_{N,\ell_{n}}) - \Delta P_{N,{n}}(\mbf A_{N,\ell_{n}}) \big)\eps_N^*(\eps_N\eps_N^*)^{p-1}.
\end{align*}
Moreover, $\eps_N^*(\eps_N\eps_N^*)^{p-1}$ is an element of $\mcal B_N$, that we denote $g(\mbf A_N)$. Since $\Delta$ is a conditional expectation, one has $\Delta\big(  P_{N,{n}}(\mbf A_{N,\ell_{n}})\big) g(\mbf A_N) = \Delta\big(  P_{N,{n}}(\mbf A_{N,\ell_{n}}) g(\mbf A_N)\big)$. Hence one can update the last coefficient of the polynomial $P_{N,n}$ in order to include the factor $\eps_N^*(\eps_N\eps_N^*)^{p-1}$. More precisely, writing 
	$$P_{N,n}=g_{0,n}(\mbf A_N)X_{k_n(1)}g_{1,n}(\mbf A_N)\cdots X_{k_n(d_n)} g_{d_n,n}(\mbf A_N),$$
	 we can replace this polynomial by
	\eq
	\tilde P_{N,n}   =  g_{0,n}(\mbf A_N)X_{k_n(1)}g_{1,n}(\mbf A_N)\cdots X_{k_n(d_n)} \tilde g_{d_n,n}(\mbf A_N),
	\qe
where $ \tilde g_{d_n,n}(\mbf A_N) =   g_{d_n,n}(\mbf A_N)g(\mbf A_N)$. Hence we are left to prove that $ \esp\left[\frac 1 N \Tr \, \eps_N \right]$ converges to zero.
\end{proof}
	
	For any $i\in \{1,...,n\}$, let $ P_{N,i} \in \mcal B_N\langle X_k:k\in K\rangle$. We define $M_i := P_{N,i}(\mbf A_{N,\ell_i})$ and $\overset \circ{M_i} :=  M_i- \Delta(M_i)$. In order to compute the quantity $\esp\big[\frac 1 N \Tr (\overset \circ{M_1}\dots \overset \circ{M_n})   \big],$
	we use the moment method, involving sums over unrooted labeled graphs that we introduce now.
	
In this article, a {\em test graph} $T$  is the data of a finite, directed and labeled graph $T=(V,E,K)$. Note that the graphs are possibly disconnected. The labelling map $K:E \to \mbf A_N$ corresponds to an assignment of matrices $e \mapsto K_e$. We define the quantities
	 \eqa
	 	\tau_N\big[ T \big] & = & \esp\bigg[\frac 1 {N^{c(T)}} \sum_{\substack{ \phi : V\to [N]}} \prod_{ e= (v,w) \in E } K _e\big( \phi(w), \phi(v) \big)\bigg]\label{Trace}\\
		\tau^0_N\big[ T \big] & = &  \esp\bigg[\frac 1 {N^{c(T)}} \sum_{\substack{ \phi : V\to [N] \\ \mrm{injective}}}  \prod_{ e= (v,w) \in E } K _e\big( \phi(w), \phi(v) \big)\bigg],
	\qea
where $c(T)$ is the number of connected components of $T$.

For any partition $\pi$ of $V$, we define $T^\pi=(V^\pi, E^\pi, K^\pi)$ as the graph obtained from $T$ by identifying vertices in the same block of $\pi$. More precisely, the vertices of $T^\pi$ are the blocks of $\pi$, each edge $e=(v,w)$ induces an edge $e^\pi=(B_v,B_w)$ where $B_v$ and $B_w$ are the blocks of $\pi$ containing $v$ and $w$ respectively. The label $K^\pi_{e_\pi}$ is the same as the label $K_e$. We say that $T^\pi$ is a quotient of $T$.  We then have the two relations \cite[Lemma 2.6]{Male2011}
	\eqa
		\tau_N\big[ T \big] & = & \sum_{\pi \in \mcal P(V)} N^{c(T^\pi) -c(T)} \tau^0_N\big[ T^\pi \big]\label{RelTrace}\label{InjTrace}\\
		\tau^0_N\big[ T \big] & = & \sum_{\pi \in \mcal P(V)} N^{c(T^\pi) -c(T)}\mrm{Mob}(0,\pi)\, \tau_N\big[ T^\pi \big],\label{RelInjTrace}
	\qea
where $\mrm{Mob}(0,\pi)=\Big(\prod_{B \in \pi}(-1)^{|B|}(|B|-1)!\Big)$ is the M\"obius function on the poset of partitions \cite[Example 3.10.4]{Sta}.
	
	We write $M_i =g_{0,i}(\mbf A_N)A_{k_i(1)}g_{1,i}(\mbf A_N)\cdots A_{k_i(d_i)} g_{d_i,i}(\mbf A_N) $, and consider the graph monomial $t_i \in \mcal F$ such that $t_i(\mbf A_N) = M_i$. It is explicitly given by
$$t_i:= \Bigg( \ \overset{\overset{g_{0,i}}\vee }\cdot \overset{( \ell_i,k_i(1))} \longleftarrow  \overset{\overset{g_{1,i}}\vee }\cdot  \overset{( \ell_i,k_i(2))} \longleftarrow   \overset{\overset{g_{2,i}}\vee }\cdot \dots \overset{\overset{g_{d_i-1,i}}\vee } \cdot   \overset{( \ell_i,k_i(d_i))} \longleftarrow \overset{\overset{g_{d_i,i}}\vee }\cdot \ \Bigg),$$
obtained by considering a directed simple path oriented from the right to the left whose edges are labelled by $( \ell_i,k_i(d_i)),\ldots, ( \ell_i,k_i(1))$ and attaching, following the decreasing order on $j=d_i, ..., 0$, the root of the graph monomial $g_{j,i}$ on the vertices of the path.

Let $T=(V,E,K)$ be the test graph obtained by identifying the output vertex of $t_i$ with the input vertex of $t_{i-1}$ for $i\in \{1\etc n\}$ (with notation modulo $n$ for the index $i$). It inherits the edge labels from the $t_i$. Note that as an unrooted graph, $T$ equals $\Delta( t_1 \cdots t_n)$. One can easily verify that $ \esp\big[\frac 1 N \Tr  [ M_1\dots M_n   ] \big] = \tau_N\big[ T\big].$

For each $i=1\etc n$, the graph $t_i$ can be seen as a subgraph of $T$. We denote by $v_i$ the vertex of $T$ corresponding to the input of $t_i$ with the convention $v_0=v_n$. The following lemma describes the cancellations in Formula (\ref{InjTrace}) when one replaces $ M_i$ by $\overset \circ{M_i}= M_i - \Delta( M_i)$. 

\begin{Lem}\label{MainLem1} With above notations for the test graph $T$ we have
	$$  \esp\left[\frac 1 N \Tr \big[ \ \!\!\overset \circ{M_1} \dots \overset \circ{M_n}\ \!\!   \big] \right] = \sum_{ \substack{ \pi \in \mcal P(V) \mrm{ \ s.t.}\\ v_i \not \sim_\pi v_{i-1}, \,  \forall i}} \tau_N^0\big[ T^\pi\big].$$
\end{Lem}
\begin{proof}

Let  us consider $I\subset\{1\etc n\}$. Denote by $T_I$ the test graph obtained form $T$ by identifying, for each $i\in I$, the input and output of $t_i$. Then we have 
	$$\esp\left[\frac 1 N \Tr \big[ \ \!\!\overset \circ{M_1} \dots \overset \circ{M_n}\ \!\!   \big]  \right]  = \sum_{ I \subset \{1\etc n\}} (-1)^{|I| } \tau_N\big[ T_I  \big].$$
For each $i=1\etc n$, we denote by $V_i$ the vertex set of $t_i$, seen in the graph $T$.
	\eq
		\esp\left[\frac 1 N \Tr \big[ \ \!\!\overset \circ{M_1} \dots \overset \circ{M_n}\ \!\!   \big]  \right]  =   \sum_{ I \subset \{1\etc n\}} (-1)^{|I|} \ \sum_{ \substack{ \pi \in \mcal P(V)\\ \mrm{s.t. \ } v_{i-1} \sim_{\pi} v_i, \, \forall i\in I} }\tau_N^0\big[ T^\pi \big]
	\qe
For any $\pi \in \mcal P(V)$, denote by $J_\pi$ the set of indices $i\in \{1\etc n\}$ such that $v_{i-1} \sim_{\pi} v_i$. Then, exchanging the order of the sums, we get
	\eq	\esp\left[\frac 1 N \Tr \big[ \ \!\!\overset \circ{M_1} \dots \overset \circ{M_n}\ \!\!   \big]  \right]   =    \sum_{ \substack{\pi  \in \mcal P(V)} }  \Big( \sum_{ \substack{I \subset J_\pi }} (-1)^{|I|}\Big)  \tau_N^0\big[ T^\pi \big].
	\qe
Since the sum in the parentheses vanishes as soon as the index  set $J_\pi$ is nonempty, we get the expected result.
\end{proof}

Our analysis of each term $\tau_N^0\big[ T^\pi\big]$ relies on the geometry of a graph we introduce now. In this article, we call colored component of $T^\pi$ a connected maximal subgraph of $T^\pi$ whose edges are labelled in one of the sets $\{1\}\times K \etc  \{L\}\times K$. We denote by $\mcal {C  C}(T^\pi)$ the set of colored components. Hence a colored component has labels corresponding to only one of the families $\mbf A_{N,1} \etc \mbf A_{N,L}$.  We call graph of colored components of $T^\pi$ the undirected bipartite graph $ \mcal {GCC}( T^\pi) = (\mcal V_\pi, \mcal E_\pi)$, where $\mcal V_\pi =\mcal {CC}(T^\pi) \sqcup V_\pi$ and there is an edge between each vertex and the colored components it belongs to. The definition is slightly different from the one in \cite{Male2011} where $\mcal V_\pi$ is the union of the colored components and the vertices of $T^\pi$ that belong to more than one colored component.

\begin{Lem}\label{MainLem21}
There is no partition $ \pi$ in the sum of Lemma \ref{MainLem1} such that $\mcal G \mcal C \mcal C\big( T^\pi\big)$ is a tree.
\end{Lem}

\begin{proof} Consider the simple cycle $c$ of $T$ that visits the edges of $t_1 \etc t_n$ labeled with $b$ in $\bigsqcup_\ell \mcal A_\ell$. For any partition $\pi$ of $\mcal P(V)$, the cycle induces a cycle $c^\pi$ (non simple in general) on the graph of colored components of $T^\pi$. If this graph is a tree, then eventually $c^\pi$ backtracks and  $\pi$ must identify $ v_i$ with  $v_{i-1}$ for some $ i\in \{1,\ldots,n\}$.
\end{proof}

The conclusion so far is that to prove the converge to zero of any $\eps_N$ as in Lemma~\ref{MainLem0}, it suffices to prove the convergence to zero of $\tau_N^0[T^\pi]$, as in Lemma~\ref{MainLem1}, for any $T^{\pi}$ such that $\mcal{GCC}(T^\pi)$ is not a tree.

\subsection{Proof of Proposition 2.3}\label{Sec:mainproof}

Henceforth, we assume that the families of matrices $\mbf A_{N,2}  \etc  \mbf A_{N,L}$ are permutation invariant and we are operating under the assumptions of Proposition \ref{prop:boundsum}. Notations are those of the previous section.

We first prove that, without loss of generality, we can assume that  $\mbf A_{N,1}$ is also permutation invariant. By Lemma \ref{Lem:FNCN}, the quantity $\esp\left[ \frac 1 N \Tr \,  \eps_N \right]$ is a linear combination of terms of the form $\esp\big[ \frac 1N \Tr \, g(\mbf A_N)\big]$, for some graph monomials $g$. Let $U$ be a uniform permutation matrix, independent of $\mbf A_N$. By \cite[Lemma 1.4]{Male2011}, for any graph monomial $g$ one has $\esp\big[ \frac 1 N\Tr \, g(\mbf A_N)\big] = \esp\big[ \frac 1N \Tr \,  g(U\mbf A_NU^t)\big]$. Since the families of matrices are independent, we have equality in distribution
	$$\big(U\mbf A_{N,1}U^t \etc U\mbf A_{N,L}U^t\big)\overset{Law}=\big(U\mbf A_{N,1}U^t,\mbf A_{N,2}  \etc  \mbf A_{N,L}\big).$$
This proves that we can replace $ \mbf A_{N,1} $ by the permutation invariant family $U\mbf A_{N,1}U^t$.

\begin{Lem}\label{MainLem2}
For any partition $\pi$ of $V$ such that $\mcal G \mcal C \mcal C( T^\pi)$ is not a tree, we have $\tau_N^0\big[ T^\pi\big] =O(N^{-1})$.
\end{Lem}

\begin{proof} Let $\phi$ be an arbitrary injective function from $V^\pi$ to  $[N]$. By permutation invariance, we have \cite[Lemma 4.1]{Male2011}
\begin{align}
\label{eq:tau0_1}
\tau^0_N\big[ T^\pi \big]  = \frac 1 N \frac{ N!}{(N-|V^\pi|)!} \esp\left[ \prod_{ e= (v,w) \in E^\pi} K^\pi_e\big( \phi(w), \phi(v) \big)\right].
\end{align}
For any $1\leq \ell\leq L$, we denote by $T^\pi_\ell=(V_\ell,E_\ell)$ the test graph which is composed of the colored components of $T^\pi$ labelled by the matrices $\mbf A_{N,\ell}$. Then, by independence of the families of matrices: 
	$$\esp\Bigg[ \prod_{ e= (v,w) \in E^\pi} K^\pi_e\big( \phi(w), \phi(v) \big)\Bigg]= \prod_{\ell=1}^L \esp\Bigg[ \prod_{ e= (v,w) \in E_\ell} K^\pi_e\big( \phi(w), \phi(v) \big)\Bigg].$$
Hence, using again \eqref{eq:tau0_1} for each graph $T^\pi_\ell$ we get
	\eq
		\tau_N^0\big[ T^\pi \big] & = & \frac 1 N \frac{ N!}{(N-|V^\pi|)!}  \ \Big(\prod_{\ell=1}^L   \frac{(N-|V_\ell|)!}{ N!}\Big) N^{|\mcal{CC}(T^\pi)|}\times \prod_{\ell=1}^L \tau^0_N\big[ T_\ell^\pi\big]\\
		& = & N^{ -1+|V^\pi| -\sum_{\ell} | V_\ell| + |\mcal{CC}(T^\pi)|}\Big( 1+ O\big(\frac 1N\big) \Big)\times \prod_{\ell=1}^L  \tau^0_N\big[ T_\ell^\pi\big],
	\qe
where $|\mcal{CC}(T^\pi)|$ is the number of colored components of $T^\pi$. Note that the cardinals of the vertex and edge sets of $\mcal G \mcal C \mcal C(T^\pi)$ are $|\mcal V_\pi| = |\mcal{CC}(T^\pi)|+ |V_\pi| $ and $|\mcal E_\pi| = \sum_{\ell} |V_\ell|$, so that 
	$$\tau_N^0\big[ T^\pi \big] = N^{ |\mcal V_\pi| - 1 -|\mcal E_\pi| } \Big( 1+ O\big(\frac 1N\big) \Big)\times \prod_{\ell=1}^L  \tau^0_N\big[ T_\ell^\pi\big].$$
A bound for $ \prod_{\ell=1}^L \tau^0_N\big[ T_\ell^\pi\big]$ is obtained from the growth condition \eqref{eq:MSbound}, for which we need the following definition.

\begin{Def}\label{Def:TEC}
\begin{enumerate}
	\item Recall that a {\em cut edge} of a finite graph is an edge whose deletion increases the number of connected components. 
	\item A {\em two-edge connected graph} is a connected graph that does not contain a cut edge. A {\em two-edge connected component} of a graph is a maximal two-edge connected subgraph. 
	\item For a finite graph $G$, we define $F(G)$ to be the graph whose vertices are the two-edge connected components of $G$ and whose edges are the cut edges linking the two-edge components that contain its adjacent vertices. Note that $F(G)$ is always a forest: we call it the \emph{forest of two-edge connected components} of $G$. 
	\item We define $\mathfrak f(G)$ to be the number of leaves of $F(G)$, with the convention that the trees of $F(G)$ that consist only of one vertex have two leaves.
\end{enumerate}
\end{Def}
\begin{Lem}\label{lem:mingo}We have the following estimate
	$$  \prod_{\ell=1}^L \tau^0_N\big[ T_\ell^\pi\big] = O\left(N^{\sum_\ell\mathfrak f(T^\pi_\ell)/2-|\mcal{CC}(T^\pi)|}\right).$$
\end{Lem}

\begin{proof} Let $\hat T=(\hat V, \hat E)$ be a test graph. By \eqref{RelInjTrace}, we have
	\eq
		\tau^0_N[\hat T] = \sum_{\sigma\in \mcal P(\hat V)} N^{c(\hat T^\sigma) - c(\hat T)} \mrm{Mob}(0,\sigma) \, \tau_N[\hat T^\sigma].
	\qe
Under the assumption of Proposition \ref{prop:boundsum}, we have $\tau_N[\hat T^\sigma] = O( N^{\mathfrak f(\hat T^\sigma)/2-c(\hat T^\sigma)})$. But $\mathfrak f(\hat T^\sigma) \leq \mathfrak f(\hat T)$, so we get $\tau^0_N[\hat T]  = O( N^{\mathfrak f(\hat T)/2-c(\hat T)})$. Applying this fact for each $T_\ell^\pi$ provides the expected result.
\end{proof}
Thus, we have the estimate
	\eqa
		\tau_N^0\big[ T^\pi \big] =O\big( N^{ |\mcal V_\pi| - 1 -|\mcal E_\pi| +\sum_\ell\mathfrak f(T^\pi_\ell)/2-|\mcal{CC}(T^\pi)| }\big).  \label{Eq:FinalEstimate}
	\qea
Let us recall that, for a finite and undirected graph $\mcal G=(\mcal V, \mcal E)$, denoting by $deg_{\mcal G}(v)$ the degree of a vertex $v\in \mcal V$, we have 
	$$ |\mcal E | - | \mcal V| = \sum_{v\in V} \Big( \frac{deg_{\mcal G}(v)}2-1 \Big).$$

Assume $\mcal G \mcal C \mcal C(T^\pi)$ is not a tree. Let $\tilde{\mcal G}_\pi = (\tilde{\mcal V}_\pi, \tilde{\mcal E}_\pi)$ be the graph obtained \emph{pruning} $\mcal G \mcal C \mcal C(T^\pi)$, by removing the vertices of $\mcal G \mcal C \mcal C(T^\pi)$ that are of degree one (as well as the edges attached to these vertices), and iterating this procedure until it does not remain vertices of degree one. We denote by $\tilde V_1$ and $\tilde V_2$ the vertices of $\mcal {CC}(T^\pi)$ and $V_\pi$  respectively that remain in $\tilde{\mcal G}_\pi$ after this process. Then we have
	\eq
		\lefteqn{ |\mcal V_\pi| - |\mcal E_\pi| - 1 +\sum_{\ell=1}^L  \frac{ \mathfrak f(T_\ell^\pi)}2 - |\mcal{CC}(T^\pi)|}\\
		& = &  |\tilde{\mcal V}_\pi| - |\tilde{\mcal E}_\pi| - 1 +\sum_{S \in \mcal C \mcal C(T^\pi)}  \Big(\frac {\mathfrak f(S)}2 -1 \Big)\\
		& = &-1 - 
		\sum_{S \in \tilde V_1}  \Big(\frac {deg_{\tilde{\mcal G}_\pi}(S)}2 -1 \Big) 
		-\sum_{v \in\tilde V_2}  \Big(\frac {deg_{\tilde{\mcal G}_\pi}(v)}2 -1 \Big)  
		+\sum_{S \in \mcal C \mcal C(T^\pi)}  \Big(\frac {\mathfrak f(S)}2 -1 \Big).
	\qe
Note that since $T^\pi$ is two-edge connected, the colored components $S$ such that $\mathfrak f(S)\geq 3$ remain in $\tilde {\mcal G}_\pi$. Indeed, since $T^\pi$ is two-edge connected, each leaf of the tree of two-edge connected components of $S$ corresponds to at least one vertex $v \in \tilde V_\pi$. So the components $S$ that have been removed have necessarily $\mathfrak f(S)=2$. Moreover, for the same reason, we have $\mathfrak f(S) \leq deg_{\tilde{\mcal G}_\pi}(S)$.  Hence, since $deg_{\tilde{\mcal G}_\pi}(v) \geq 2$ for each $v\in \tilde V_2$, the last equality yields \eq
	|\mcal V_\pi| - |\mcal E_\pi| - 1 +\sum_\ell  \frac{\mathfrak{f}(T_\ell^\pi)}{2} -|\mcal{CC}(T^\pi)|\leq -1 - \sum_{v \in\tilde V_2}  \Big(\frac {deg_{\tilde{\mcal G}_\pi}(v)}2 -1 \Big) \leq -1.
	\qe
By Formula \eqref{Eq:FinalEstimate}, we get as expected $\tau_N^0[T^\pi]\limN 0$. 
\end{proof}
\subsection{Proof of Theorem 2.2}\label{Sec:proofmain}Theorem 6 of \cite{Mingo2012} ensures that
$$\Tr \big( g(\mbf A_{N,\ell}) \big)\leq N^{\mathfrak f(g)/2}\prod_{e\in E}\|A_{N,\ell(e)}^{k(e)}\|$$
for any graph monomial $g$ with set of edges $E$ and equal input and output. Because the norm of our matrix are uniformly bounded, the matrices satisfies the growth assumption \eqref{eq:MSbound} of Proposition~\ref{prop:boundsum}, and all the results of the previous section are valid.

\subsection{Proof of Theorem 1.2}\label{Sec:proofFirstTh}

Let $\mbf A_{N,1}  \etc \mbf A_{N,L}$ and $\eps_N$ be as in Theorem \ref{FirstTh}. By enlarging the index set $K$ if necessary, let us denote by $\mbf A_{N,L+1}$ the family of the coefficients $D_{i,\ell}$ of the polynomial $P_{N,\ell}$ defining $\eps_N$. We apply Theorem \ref{MainTh} to the family $\mbf A_{N,1} \etc  \mbf A_{N,L+1}$, specified for the polynomials
	$$\tilde P_{N,\ell} = D_{0,\ell} X_{k_i(1)}D_{1,\ell} \cdots  X_{k_i(d_\ell)}D_{d_\ell,\ell}.$$
Note that each $D_{i,\ell}$ can indeed be written as $g_{i,\ell}(\mbf A_N)$ for $g_{i,\ell}$ a graph monomial consisting in a single loop labeled by the matrix $D_{i,\ell}$ of $\mbf A_{N,L+1}$. This proves the assertion of Theorem \ref{FirstTh} for any $P_{N,\ell}$ with fixed degree and bounded coefficients, which is sufficient to ensure the validity of the theorem in full generality.

\subsection{Proof of Proposition 2.4}

Let us consider the families of random matrices $\tilde{\mathbf{A}}_{N,1},...,\tilde{\mathbf{A}}_{N,L}$,
where we recall that $\tilde{\mathbf{A}}_{N,\ell}:=\big(A_{N,\ell}^{(k)}\circ\Gamma_{\ell}^{(k)}\big)_{k\in K}$,
and $\Gamma=\big(\Gamma_{\ell}^{(k)}\big)_{\ell,k}$ is a family of random matrices, with uniformly bounded entries, independent of the matrices $(\mathbf{A}_{N,1},...,\mathbf{A}_{N,L})$.

According to Section~\ref{Sec:PrelLem}, in order to have the conclusion of Theorem~\ref{MainTh}, it suffices to prove that $\tau_N^0\big[ T^\pi(\tilde {\mbf A}_N)\big]$ converges to $0$ for all test graph $T$ and partition $\pi$ such that $\mcal{GCC}(T^\pi)$ is not a tree. Using \eqref{eq:tau0_1}, we have
\begin{align*}
&\tau_N^0\big[ T^\pi(\tilde {\mbf A}_N)\big]\\
 &=\frac{1}{N} \sum_{\substack{ \phi : V\to [N] \\ \mrm{injective}}}\esp \Big[\prod_{e = (v,w)\in E} A_{\ell(e)}^{k(e)}\big( \phi(w) , \phi(v)\Big] \times \esp \Big[\prod_{e = (v,w)\in E} \Gamma_{\ell(e)}^{k(e)}\big( \phi(w) , \phi(v)\Big]\\
&=\tau_N^0\big[ T^\pi(  {\mbf A_N})\big]\cdot  \frac{(N-|V^\pi|)!}{N!}\sum_{\substack{ \phi : V\to [N] \\ \mrm{injective}}}\esp \Big[\prod_{e = (v,w)\in E} \Gamma_{\ell(e)}^{k(e)}\big( \phi(w) , \phi(v)\Big].
\end{align*}The term $\tau_N^0\big[ T^\pi(  {\mbf A_N})\big]$ tends to $0$ thanks to Section~\ref{Sec:mainproof} and Section~\ref{Sec:proofmain}, and the rightmost term is bounded, since the entries of the matrices $\Gamma_{\ell}^{(k)}$ are bounded. Thus the conclusion of Theorem~\ref{MainTh} are true for $\tilde{\mathbf{A}}_{N,1},...,\tilde{\mathbf{A}}_{N,L}$.

Reasoning as in Section~\ref{Sec:proofFirstTh}, in order to have the conclusion of Theorem~\ref{FirstTh} true for $\tilde{\mathbf{A}}_{N,1},...,\tilde{\mathbf{A}}_{N,L}$, it suffices to have the conclusion of Theorem~\ref{MainTh} true for $\tilde{\mathbf{A}}_{N,1},...,\tilde{\mathbf{A}}_{N,L}, \tilde{\mathbf{A}}_{N,L+1}$, where $\tilde{\mathbf{A}}_{N,L+1}$ is a family of bounded determistic diagonal matrices. We can write $\tilde{\mathbf{A}}_{N,L+1}$ as $\big(A_{N,\ell}^{(k)}\circ\Gamma_{\ell}^{(k)}\big)_{k\in K}$ with $A_{N,L+1}^{(k)}=I_N$ and $\Gamma_{L+1}^{(k)}=\tilde{A}_{N,L+1}^{(k)}$, in such a way that $\mathbf{A}_{N,L+1}=(A_{N,L+1}^{(k)})_{k\in K}$ is permutation invariant. Then we have proved that the conclusion of Theorem~\ref{MainTh} is true for $\tilde{\mathbf{A}}_{N,1},...,\tilde{\mathbf{A}}_{N,L+1}$.

\bibliographystyle{plain}
\bibliography{biblio}
\end{document}